\documentclass[12pt]{amsart}
\usepackage{amsfonts}
\usepackage{amsmath}
\usepackage{amsxtra}
\usepackage{amssymb,latexsym}
\usepackage[mathcal]{eucal}
\usepackage{amscd}
\usepackage[pdftex,bookmarks,colorlinks,breaklinks]{hyperref}
\input xy
\xyoption{all}

\newcommand{\Acal}{\mathcal{A}}

\newcommand{\Fcal}{\mathcal{F}}

\newcommand{\Hcal}{\mathcal{H}}

\newcommand{\Ocal}{\mathcal{O}}

\newcommand{\Ucal}{\mathcal{U}}

\newcommand{\Xcal}{\mathcal{X}}

\newcommand{\ch}{\mathbf{1}}

\newcommand{\Z}{\mathbb{Z}}

\newcommand{\C}{\mathbb{C}}
\newcommand{\N}{\mathbb{N}}
\newcommand{\T}{\mathbb{T}}
\newcommand{\E}{\mathbb{E}}

\newcommand{\OC}{\bar{\mathcal{O}}}

\newcommand{\Ga}{\Gamma}
\newcommand{\ga}{\gamma}
\newcommand{\del}{\delta}

\newcommand{\ep}{\epsilon}
\newcommand{\sig}{\sigma}

\newcommand{\la}{\lambda}
\newcommand{\La}{\Lambda}

\newcommand{\Om}{\Omega}

\newcommand{\ol}{\overline}

\newcommand{\br}{\vspace{3 mm}}
\newcommand{\imp}{\Rightarrow}

\newcommand{\rest}{\upharpoonright}

\newcommand{\cls}{{\rm{cls\,}}}

\newcommand{\Aut}{{\rm{Aut\,}}}

\newcommand{\eva}{{\rm{eva}}}

\newcommand{\spann}{{\rm{span}}}

\newcommand{\Iso}{{\rm{Iso\,}}}
\newcommand{\Homeo}{{\rm{Homeo\,}}}

\newcommand{\onto}{\twoheadrightarrow}

\swapnumbers \theoremstyle{plain}
\newtheorem{thm}{Theorem}[section]
\newtheorem{cor}[thm]{Corollary}
\newtheorem{lem}[thm]{Lemma}
\newtheorem{prop}[thm]{Proposition}

\theoremstyle{definition}

\newtheorem{defns}[thm]{Definitions}

\newtheorem{prob}[thm]{Problem}

\numberwithin{equation}{section}

\begin{document}
\title[On Hilbert dynamical systems]
{On Hilbert dynamical systems}

\author{Eli Glasner and Benjamin Weiss}

\address{Department of Mathematics\\
     Tel Aviv University\\
         Tel Aviv\\
         Israel}
\email{glasner@math.tau.ac.il}

\address {Institute of Mathematics\\
 Hebrew University of Jerusalem\\
Jerusalem\\
 Israel}
\email{weiss@math.huji.ac.il}

\begin{abstract}
Returning to a classical question in Harmonic Analysis 
we strengthen an old result
of Walter Rudin. We show that there exists a weakly almost periodic function
on the group of integers $\Z$ which is not in the norm-closure
of the algebra $B(\Z)$ of Fourier-Stieltjes transforms of measures on the dual
group $\hat{\Z}=\T$, and which is {\bf recurrent}. 
We also show that there is a Polish monothetic group which is reflexively but 
not Hilbert representable.
\end{abstract}

\thanks{{Both author's work is supported by ISF grant 
\#1157/08 and the first author's work is also supported
by BSF grant \# 2006119}}

\date{February 10, 2010}

\keywords{}

\maketitle

\tableofcontents

\section*{Introduction}

Walter Rudin \cite{Ru} was the first to show that the algebra $WAP(\Z)$, 
of weakly almost periodic functions on $\Z$, is strictly larger than 
the algebra $H(\Z)= \ol{B(\Z)}$, the norm closure in the Banach space 
$\ell_\infty(\Z)$ of the algebra $B(\Z)$ of Fourier-Stieltjes transforms of 
complex measures on the dual group $\hat{\Z} = \T$. 
Many other examples of functions in $WAP(\Z) \setminus H(\Z)$ followed 
(see e.g. \cite{DR}).

As it turns out, in all of these examples the function in question is non-recurrent.
We say that a function $f \in \ell_\infty(\Z)$ is {\em recurrent} if there is
a sequence $n_k \nearrow \infty$ with 
$\lim_{k \to \infty}\sup_{j \in \Z} |f(j) - f(j +n_k)| =0$.
Now in some sense the more interesting functions in $WAP(\Z)$ are the 
recurrent ones, and moreover they are prevalent; e.g.
every $\Z$-subalgebra of $WAP(\Z)$ which is $\Z$-generated by a recurrent
function consists entirely of recurrent functions
(see Lemma \ref{Rec-WAP} below). Also, a Fourier-Stieltjes transform
$\hat{\mu}$ of a continuous measure $\mu$ on $\T$ is recurrent iff
$\mu$ is Dirichlet \cite{HMP}.
Thus the question
whether there are {\em recurrent} functions in $WAP(\Z) \setminus H(\Z)$
is a natural one.

Another related open question 
is the following. Is there a Polish monothetic group $P$ which
can be represented as a group of isometries of a reflexive Banach space
but is not representable as a group of unitary operators on a Hilbert space?

In this work we see how these questions are related
and show that the answer to both is affirmative
(Theorem \ref{WAPH-1} and Corollary \ref{rf}). 
We also show that if a Polish monothetic group $P$ is Hilbert-representable
and $K \subset P$ is a compact subgroup, then
the quotient group $P/K$ is also Hilbert-representable 
(Corollary \ref{AK-ext-cor}).

Our proofs are based on the theory of topological dynamics and rely on
a well known construction of Banaszczyk \cite{Ba}. We also use
an idea of Megrelishvili \cite{Me-00} who showed that
the topological groups $L_{2n}([0,1])$, for $n$ a positive integer, are
reflexively but not Hilbert representable.
For the sake of simplicity we state our results in the most
basic setup, where the acting group is $\Z$ and the dynamical
systems are usually assumed to be point-transitive.

We refer the reader to the following related recent works:
Gao and Pestov \cite{Gao-P}, Megrelishvili \cite{Me},
Ferri and Galindo \cite{Ga-F}, and Galindo \cite{Ga}.

\section{Some preliminaries from topological dynamics}

A {\em dynamical system} (or sometimes just a {\em system}) 
is for us a pair $(X,T)$ where $X$ is a compact
Hausdorff space and $T : X \to X$ is a self homeomorphism.
With $(X,T)$ we associate an action of the group of integers $\Z$ 
via the map $n \mapsto T^n$. The {\em orbit} of a point $x \in X$ is the set
$\Ocal_T(x) =\{T^n x : n \in \Z\}$. The {\em orbit closure} of $x$
is the set $\OC_T(x) = \ol{\Ocal_T(x)}$. The system $(X,T)$ is 
{\em point-transitive} if there is a point $x \in X$ with $\OC_T(x)= X$.
Such a point is called a {\em transitive point} and the collection of
transitive points (when it is not empty) is denoted by $X_{tr}$.
For a point-transitive metric system $X_{tr}$ is a dense $G_\del$
subset of $X$. We will mostly work in the category of {\em
pointed} dynamical systems $(X,x_0,T)$, where the latter is 
a point-transitive dynamical system with a distinguished point
$x_0 \in X_{tr}$. 
The restriction of $T$ to a closed invariant subset $Y \subset X$ 
in a dynamical system $(X,T)$ defines a dynamical system $(Y,T)$.
Such a system is called a {\em subsystem} of $(X,T)$. 

A continuous surjective map 
$\pi: (X,T) \to (Y,S)$ between two dynamical systems $(X,T)$ and $(Y,S)$
which intertwines the $\Z$-actions (i.e. $\pi(Tx) = S\pi(x)$ for every
$x \in X$) is called a {\em homomorphism} of dynamical systems
and we sometimes say that $Y$ is a {\em factor} of $X$ or that 
$X$ is an {\em extension} of $Y$.
When dealing with pointed systems a homomorphism
$\pi: (X,x_0,T) \to (Y,y_0,S)$ is further assumed to satisfy $\pi(x_0) = y_0$.
An extension $\pi : (X,T) \to (Y,S)$ is called {\em almost 1-1} if there is
a dense $G_\del$ subset $X_0 \subset X$ with 
$\pi^{-1}(\pi(x)) =\{x\}$ for every $x \in X_0$. The extension
$\pi: (X,x_0,T) \to (Y,y_0,S)$ is called a {\em group-extension} if there is 
a compact subgroup $K \subset \Homeo(X)$ such that each $k \in K$
is an automorphism of $(X,T)$ (i.e. $Tk = kT$ for every $k \in K$) 
and such that the quotient dynamical system $(X/K, T)$ is isomorphic to 
$(Y,S)$.

A point $x$ in a metric dynamical system $(X,T)$ is called
{\em recurrent} if there is a sequence $\{n_k\} \subset \Z$
with $|n_k| \to \infty$ such that $\lim T^{n_k} x = x$. 
Note that in a point-transitive system if the set of isolated points is not
empty then it coincides with the orbit of a transitive point.
On the other hand, when there are no isolated points in a point 
transitive system then every point of $X_{tr}$ is recurrent.
We call a point-transitive system with no isolated points a {\em 
recurrent-transitive system}. 

\br

Let $\ell_\infty(\Z)$ be the Banach space (and $C^*$-algebra) of bounded complex valued functions on $\Z$ with the sup norm: 
$\|f\|_\infty= \sup_{n \in \Z} |f(n)|$.
We write $S: \ell_\infty(\Z) \to \ell_\infty(\Z)$ for the shift operator, where
$Sf(n) = f(n+1),\ (n \in \Z)$.
An $S$-invariant, conjugacy invariant subalgebra of $\ell_\infty(\Z)$
containing the constant function $\ch$ will be called a {\em $\Z$-algebra}.

Given a pointed dynamical system $(X,x_0,T)$ we define a map
$j_{x_0}: C(X) \to \ell_\infty(\Z)$ by
$$
j_{x_0}F(n) = F(T^n x_0), \quad F \in C(X),\ n \in \Z.
$$
It is easy to see that $j_{x_0}$ is an isometry with $j_{x_0}\circ T
=S \circ j_{x_0}$. We denote its image in $\ell_\infty(\Z)$ by $\Acal(X,x_0)$.
Clearly $\Acal(X,x_0)$ is a $\Z$-algebra.

Conversely, given a $\Z$-algebra $\Acal\subset \ell_\infty(\Z)$
we denote its Gelfand space (comprising the non-zero $C^*$-homomorphisms
of $\Acal(X,x_0)$ into $\C$) by $X=|\Acal(X,x_0)|$. It is easy
to see that the operator $S$ induces a homeomorphism 
$T : X \to X$ and that the resulting dynamical system $(X,T)$
is a point-transitive system. In fact, the point $x_0=\eva_0 \in X$, which 
corresponds to the multiplicative complex valued homomorphism 
defined on $\Acal$ by evaluation at $0$, is a transitive point.

These operations are inverse to each other and we have
$$
(|\Acal(X,x_0)|,\eva_0,S) \cong (X,x_0,T).
$$
A $\Z$-algebra $\Acal$ is {\em cyclic} if there is a function $f \in \Acal$
such that $\Acal = \Acal_f$, where the latter is the smallest $\Z$-algebra
that contains $f$. 

Given $f \in \ell_\infty(\Z)$ we can consider $f$ as an element
of the compact metrizable space $\Om=[-\|f\|_\infty,\|f\|_\infty]^\Z$.
Again denote by $S : \Om \to \Om$ the homeomorphism
defined by $Sg(n) = g(n+1),\ (g \in \ell_\infty,\ n \in \Z)$.
We let $X_f = \cls\{S^n f : n \in \Z\}$, where the closure is taken in 
$\Om$ with respect to the product topology.
It can be easily verified that $\Acal(X_f,f,S) =\Acal_f$.
Note that as $\Acal_f$ is always separable, $X_f$ is metrizable.

If $\pi:(X,x_0,T)\to (Y,y_0,S)$ is a homomorphism
of pointed transitive systems then the diagram
\begin{equation*}
\xymatrix
{
C(Y) \ar[d]_{\pi^*} \ar[r]^{j_{y_0}} & \Acal(Y,y_0)
\ar[d]^i \\
C(X) \ar[r]_{j_{x_0}} & \Acal(X,x_0)
}
\end{equation*}
commutes. Here $(\pi^* F)(x)=F(\pi x)$, for
$F\in C(Y), x\in X$, and $i$ is the inclusion map.
Conversely, if $B\subset C(X)$ is a closed conjugation-invariant,
$T$-invariant subalgebra containing the constant functions, then
the restriction map
$$
\pi: (|\Acal(X,x_0)|,{\eva}_0,T) \to
(|j_{x_0}(B)|,{\eva}_0,T)
$$
is a pointed homomorphism of dynamical systems.

\br

With every dynamical system $(X,T)$ we associate its 
{\em enveloping semigroup} $E=E(X,T) \subset X^X$. This is 
the pointwise closure of the set $\{T^n : n \in \Z\}$, as a subset of $X^X$.
$E(X,T)$ is a compact {\em right topological semigroup}, i.e.
for each $p \in E(X,T)$ right multiplication $q \mapsto qp,\ q \in E$
is continuous. The set $\{T^n : n \in \Z\}$ is contained in the center of $E$.
In particular, for each $n \in \Z$ the map $p \mapsto T^n x,\  p \in E$
is a homeomorphism of $E$, so that via multiplication by $T$,
$(E,T)$ is also a dynamical system.

\br

A function $f \in \ell_\infty(\Z)$ is an {\em almost periodic}
({\em weakly almost periodic}) function if its orbit
$\{S^n f: n \in \Z\}$ is norm precompact (weakly precompact)
in the Banach space $\ell_\infty(\Z)$. 
We denote by $AP(\Z)$ and $WAP(\Z)$
the collections of almost periodic and weakly almost periodic functions
in $\ell_\infty(\Z)$ respectively. These are $\Z$-algebras with $AP(\Z)
\subset WAP(\Z)$. A point-transitive dynamical system 
$(X,T)$ is called {\em almost periodic} iff $\Acal(X,x_0,T) \subset AP(\Z)$
iff for every $F$ in the Banach space $C(X)$ the orbit
$\{F \circ T^n : n \in \Z\}$ is norm pre-compact. 
The system $(X,T)$ is {\em weakly
almost periodic} iff $\Acal(X,x_0,T) \subset WAP(\Z)$ iff 
for every $F \in C(X)$ the orbit $\{F \circ T^n : n \in \Z\}$ is weakly 
pre-compact. 
It is well known that a point-transitive system $(X,T)$ is almost periodic 
iff it is equicontinuous and minimal. A theorem of 
Ellis and Nerurkar \cite{EN} based on a theorem of Grothendieck
asserts that $(X,T)$ is weakly almost periodic iff
its {\em enveloping semigroup} $E(X,T)$ consists
of continuous maps. Another characterization of WAP systems
(which again goes back to Grothedieck) is that $E(X,T)$ be
a commutative semigroup. Using this observation it is also easy
to deduce the following well known {\em double limit criterion}
(stated here for a general topological group $G$, see e.g. \cite{Rupp}).

\begin{prop}\label{DLC}
For a topological group $G$, a bounded continuous
function $f : G \to \C$ is WAP iff whenever $g_m, h_n$ are 
sequences in $G$ such that the double limits
$$
a=\lim_{n \to \infty}\lim_{m \to \infty} f(g_m h_n)
\quad {\text{and}}\quad
b=\lim_{m \to \infty}\lim_{n \to \infty} f(g_m h_n)
$$
exist, then $a=b$.
\end{prop}

For WAP $\Z$-systems we have $E(X,T) \cong (X,T)$, 
see e.g. \cite{Do} or \cite{G}.

We summarize these results in the following:

\begin{thm}\label{WAP-thm}
Let $(X,T)$ be a point-transitive dynamical system and let
$E=E(X,T)$ be its enveloping semigroup.
The following conditions are equivalent.
\begin{enumerate}
\item
The system $(X,T)$ is WAP.
\item
$E$ is a semi-topological semigroup, i.e. both right and left
multiplications are continuous.
\item
$E$ is a commutative semigroup.
\item
$E$ consists of continuous maps.
\end{enumerate}
\end{thm}

\br

A point $x$ in a metric dynamical system $(X,T)$ is an 
{\em equicontinuity point} if for every $\ep>0$ there is a $\del>0$
such that $d(x,x')< \del \imp d(T^nx, T^nx')< \ep$ for all $n \in \Z$.
A dynamical system with a residual set of equicontinuity points
is called an {\em almost equicontinuous} system. A recurrent-transitive metric
almost equicontinuous system $(X,T)$ is uniformly rigid
and the the set of equicontinuity points in $X$ coincides with $X_{tr}$.
Moreover, the map $\La_X \onto X_{tr},\ \la \mapsto \la x_0$,
is a homeomorphism of the Polish group $\La_X$ onto the dense
$G_\del$ subset $X_{tr} \subset X$ \cite{GW2}.

By results of Glasner and Weiss \cite{GW} and
Akin, Auslander and Berg \cite{AAB} it follows that if $(X,T)$ is a metric
recurrent-transitive WAP dynamical system, then
the system $(X,T)$ is {\em uniformly rigid}, i.e. there is a
sequence  of positive integers $n_k \to \infty$ such that the sequence
$\{T^{n_k}: k \in \N\}$ tends uniformly to the identity map on $X$
(see \cite{GMa}).
In a uniformly rigid system the uniform closure 
$$
\La_X = \cls \{T^n: n \in \Z \} \subset \Homeo(X),
$$
is a non-discrete Polish monothetic topological group. We then have
$X_{tr} = \La_X x_0$.

When $(X,T)$ is recurrent-transitive metric WAP we know much more: 
The system $(X,T)$
is {\em hereditarily almost equicontinuous}; 
i.e., every subsystem $Y \subset X$ is almost equicontinuous 
(see \cite{GW}, \cite{AAB}, \cite{GW2} and \cite{GM}).

\section{Reflexive and Hilbert representability of groups
and dynamical systems}

\begin{defns}
\begin{enumerate}
\item
We say that a Polish topological group $P$ is {\em representable} 
on a reflexive Banach space $V$ if there is a topological isomorphism 
of $P$ into the group of isometries $\Iso(V)$ of $V$ equipped with the 
strong operator topology. 
The group $P$ is {\em reflexibly-representable} 
it is representable on some reflexive Banach space.
\item
The group $P$ is {\em Hilbert-representable} if there is a topological 
isomorphism of $P$ into the unitary group $\Ucal(\Hcal)$ of a Hilbert space 
$\Hcal$ equipped with the strong operator topology (see Megrelishvili \cite{M1}
and \cite{Me}).
\item 
As in Banaszczyk \cite{Ba} we say that $P$ is 
{\em exotic} if it does not admit any nontrivial continuous unitary 
representations. 
We say that $P$ is {\em strongly exotic} if it does not admit any nontrivial
weakly continuous representations in Hilbert spaces.
\end{enumerate}
\end{defns}

\begin{defns}
\begin{enumerate}
\item
A 
metric dynamical system $(X,T)$ is called  
{\em reflexively-repre\-sentable} if 
there is a reflexive Banach space $V$, a linear isometry
$U \in Iso(V)$  and a weakly compact $U$-invariant
subset $Z$ of $V$ such that the dynamical systems
$(X,T)$ and $(Z,U)$ are isomorphic.
\item
A 
metric dynamical system $(X,T)$ is called 
{\em Hilbert-representable} if 
there is a Hilbert space $\Hcal$, a unitary operator $U \in \Ucal(\Hcal)$
and a weakly compact  $U$-invariant subset $Z$ of  $\Hcal$ such that the dynamical systems $(X,T)$ and $(Z,U)$ are isomorphic.
\item
Let $B(\Z)$ be the sub-algebra of $\ell_\infty(\Z)$ which consists of
the Fourier-Steiltjes transforms of complex measures on the circle $\T$,
i.e. $B(\Z) = \Fcal (M(\T))$ where $\Fcal: M(\T) \to \ell_\infty(\Z)$
is the Fourier transform:
$$
\Fcal(\mu)(n) = \hat\mu(n) = \int_{\T} e^{int} \,d\mu(t).
$$
\item
Let $H(\T)=\ol{B(\Z)}$ be the norm closure of $B(\Z)$ in the Banach
space $\ell_\infty(\Z)$. Clearly $H(\Z)$ is a $\Z$-subalgebra of $\ell_\infty(\Z)$.
\item
A point-transitive system $(X,x_0,T)$ is called {\em Hilbert} if
$\Acal(X,x_0) \subset H(\Z)$. The elements of $H(\Z)$ are called 
{\em Hilbert functions}.
\end{enumerate}
\end{defns}

Suppose $(X,T)$ is reflexively-representable. Thus we assume that
there is a reflexive Banach space $V$, a linear isometry $U \in \Iso(V)$
and that $X \subset V$ is a weakly compact $U$-invariant
subset, with $T=U\rest X$.
Let $x_0$ be a vector in $X \subset V$ and $\phi$ a vector in $V^*$. 
Set $F(x) =\langle x, \phi\rangle$ 
and $f(n) = F(T^nx_0)=\langle T^nx, \phi\rangle$ 
(such a function is called a {\em matrix coefficient} of the representation).
Then, it is easy to see how the weak-compactness of the
unit ball of $V^*$ implies that the $T$-orbit of $F$ in $C(X)$ is
weakly precompact. 
Since $\Acal(X,x_0)\cong C(X)$ it follows
that the $S$-orbit of $f$ in $\ell_\infty(\Z)$ is also weakly precompact, 
i.e. the function $f$ is in $WAP(\Z)$. 

It turns out that the converse is also true.
We have the following basic theorems of Shtern \cite{S}
and Megrelishvili \cite{M1} concerning
reflexive representability.
We formulate these results in the context of a general topological group $G$,
where the $C^*$-algebra $LUC(G)$ of bounded, complex valued, 
left uniformly continuos functions takes the place of $\ell_\infty(\Z)$.
Thus e.g. $WAP(G)$ is the
$G$-subalgebra of $LUC(G)$ comprising functions
whose $G$-orbit is weakly pre-compact.

\begin{thm}\label{S-M}
\begin{enumerate}
\item 
A topological group $G$ can be faithfully represented
on a reflexive Banach space $V$ iff the $WAP(G)$ functions
separate points and closed sets on $G$ \cite{S}, \cite{M1}.  
%
\item 
Let $G$ be a topological group, then every $f \in WAP(G)$ is a 
matrix coefficient of a reflexive representation of $G$ \cite{M1}.
\item  
A metrizable $G$-system $(X,G)$ is WAP iff it is reflexively-representable
\cite{M1}.  
\end{enumerate}
\end{thm}

\br

\begin{lem}\label{HR}
If a point-transitive metric system $(X,x_0,T)$ is Hilbert-representable 
then there is a function $F \in C(X)$ such that the corresponding 
$f \in \Acal(X,x_0)$,
i.e. the function $f \in \ell_\infty(\Z)$ defined by 
$f(n) = F(T^nx_0)$, is positive definite and satisfies
$\Acal(X,x_0) = \Acal_f$, where the latter is the $\Z$-subalgebra
generated by $f$ in $\ell_\infty(\Z)$.
\end{lem}

\begin{proof}
By assumption we can consider $X$ as a weakly compact subset
of a separable Hilbert space $\Hcal$ and $T= U\rest X$, where 
$U \in \Ucal(\Hcal)$, the unitary group of $\Hcal$.
With no loss in generality we also assume that $\|x_0\|=1$ and that
$\Hcal = Z_U(x_0)$, where the latter is the $U$-cyclic space generated
by $x_0$.
Set $F(x) = \langle x, x_0 \rangle$. Then $F\in C(X)$ and
$$
f(n) = F(U^nx_0) =  \langle U^n x_0, x_0 \rangle
$$
is indeed positive definite.

If $x \ne y$ are points in $X$ then, as $\Hcal = Z_U(x_0)$,
there is some $n \in \Z$ with 
$$
F(U^{-n}x)=\langle x, U^{-n}x_0 \rangle
\ne \langle y, U^{-n}x_0 \rangle = F(U^{-n}y),
$$
so that the sequence of functions $\{F\circ U^n\}_{n\in \Z}$ separates
points on $X$, whence $\Acal(X,x_0) = \Acal_f$.
\end{proof}

\begin{lem}\label{PD}
Suppose $f \in \ell_\infty(\Z)$ is positive definite.
\begin{enumerate}
\item 
The system $(X_f,f,S)$ is Hilbert-representable, i.e.
there exists a Hilbert space $\Hcal$, a unit vector $x_0 \in \Hcal$ and a 
unitary operator $U \in \Ucal(\Hcal)$ such that $\Hcal  = Z_U(x_0)$, 
$X = w$-$\cls \{U^nx_0 : n \in \Z\}$, and $(X_f,f,S) \cong (X,x_0,U)$.
\item
Every element $g$ of $X_f= \OC(f) \subset [-\|f\|,\|f\|]^\Z$ has the form
$g(j) =  \langle U^j x_0, x \rangle$ for some $x \in X$. 
In particular $g \in B(\Z)$. 
\item
For $g \in X_f = \OC(f)$ we have:
$g \in \Ocal(f)$ iff there is $n \in \Z$ with $g(n) = \|g\|=\|f\|$,
in which case $g = S^n f$. In particular for $g \in X_f$ we have
$g(0) = \|g\|=\|f\| \imp g =f$.
\item 
$(X_f)_{tr} = \{g \in X_f : \|g\| = \|f\|\}$.
\end{enumerate}
\end{lem}

\begin{proof}
1.\ 
This is a consequence of Herglotz' theorem. 

2.\
Let $g \in X_f$. Then there exists a sequence $n_k \to \infty$ with
$$
g(j) = \lim_{k \to \infty} f(j + n_k) = \lim_{k \to \infty} S^{n_k} f(j) =
\lim_{k \to \infty} \langle U^j x_0, U^{n_k} x_0 \rangle
=  \langle U^j x_0, x \rangle ,
$$
where we assume, with no loss in generality, that $x = w$-$\lim_{k \to\infty}
U^{n_k}x_0$ exists.
It follows that $g \in B(\Z)$.

3.\ 
Suppose $g \in X_f$ satisfies $g(0) = \|g\|=\|f\|$. As we have seen 
in part 1, there is an $x \in X$ with  $g(j) = \langle U^j x_0, x \rangle$ 
for every $j\in \Z$ and  our assumption reads:
$$
1= g(0) = \langle x , x_0 \rangle. 
$$
We conclude that $x = x_0$, hence 
$$
g(j) = \langle U^j x_0, x_0 \rangle = f(j),
$$
for all $j \in \Z$, i.e. $g = f$.

4.\ 
Let $g \in (X_f)_{tr}$. There is the a sequence $n_k$ with 
$\lim_{k \to \infty} S^{n_k} g(j)  = \lim_{k \to \infty}  g(j + n_k) = f(j)$
for every $j \in \Z$. In particular $\lim_{k \to \infty} g(n_k) = f(0) = \|f\|$.
Thus $\|f \| \le \|g\| \le \|f\|$, hence $\|f \|= \|g\|$.
Conversely, assuming $\|f \|= \|g\|$, we have 
$\lim_{k \to \infty} g(n_k) = f(0)=\|f\|$. With no loss of generality we can assume 
that $h = \lim_{k \to \infty} S^{n_k}g$ exists in $X_f$, so that $h(0) =\|f\|$.
By part 3, $h = f$ and we conclude that $g \in (X_f)_{tr}$.
\end{proof}

\begin{prop}\label{sh=pd}
A point-transitive system $(X,x_0,T)$ is Hilbert-representable iff there is
a positive definite function $f \in \ell_\infty(\Z)$ such that
$(X,x_0,T) \cong (X_f,f,S)$, where $X_f$
is the orbit closure of $f$ in $\ell_\infty(\Z)$ under the shift
$S$ with respect to the pointwise convergence topology.
In particular every Hilbert-representable system is Hilbert.
\end{prop}

\begin{proof}
Combine Lemmas \ref{HR} and \ref{PD}.
\end{proof}

\begin{prop}
Every Hilbert system is WAP, i.e. $H(\Z) \subset WAP(\Z)$.
\end{prop}

\begin{proof}
It is well known (and easy to see) that $B(\Z) \subset WAP(\Z)$.
Since the algebra $WAP(\Z)\subset \ell_\infty(\Z)$ is closed we also have 
$H(\Z) \subset WAP(\Z)$.
\end{proof}

\section{A structure theorem for Hilbert systems}

\begin{thm}\label{AK-ext}
Every metrizable recurrent-transitive Hilbert system $(Y,T)$ 
admits an almost 1-1 extension $(\tilde{Y},\tilde{T})$ which is a 
compact group-factor of a recurrent-transitive Hilbert-representable 
system $(X,U)$. 
Thus there exists a commutative diagram
\begin{equation}\label{AK}
\xymatrix
{
(X,U) \ar[dd]_{\pi}\ar[dr]^{\sig}  & \\
 & (\tilde{Y},\tilde{T})\ar[dl]^{\rho}\\
(Y,T) &
}
\end{equation}
with $\sig$ a group-extension and $\rho$ an almost 1-1 extension.
\end{thm}

\begin{proof}
1.\ 
Let $(Y,T)$ be a metrizable recurrent-transitive Hilbert system. 
Thus, picking a transitive point
$y_0 \in Y$ we have $\Acal(Y,y_0) \subset H(\Z)$. Since every function
in $B(\Z)$ is a linear combination of four positive definite functions,
we can find a sequence $f_n$ of positive definite functions
(with $f_1 \equiv \ch$ and $f_n(0) = \|f_n\|_\infty = 1\ (n=1,2,\dots)$ such that 
$\Acal(\{f_n: n \in \N\})$, the closed translation and conjugation 
invariant algebra generated by the functions $f_n$, contains $\Acal(Y,y_0)$.
For each $n$ there is a separable Hilbert space $\Hcal_n$, a unit vector
$x_n \in \Hcal_n$, and a unitary operator $U_n \in \Ucal(\Hcal_n)$ such that
$f_n(k) = \langle U_n^k x_n,x_n \rangle \ (\forall k \in \Z)$.
Let $\Hcal= \oplus_{n=1}^\infty \Hcal_n$ (an $\ell_2$-sum), 
$x_0 = \oplus_{n=1}^\infty x_n \in \Hcal$,
and $U = \oplus_{n=1}^\infty U_n \in \Ucal(\Hcal)$. Set $X = {\text{weak-}}\cls
\{U^k x_0 : k \in \Z\}$. Then it is easy to verify that the Hilbert-representable system $(X,U)$ admits $(Y,T)$ as a factor, say
$\pi: (X,U) \to (Y,T)$. 

2.\ 
Let $\Xcal$ be the collection of all subsystems $(X',U)$ of $(X,U)$ with
$\pi(X')= Y$. By Zorn's lemma there is a minimal element 
$(Z,U)$ in $\Xcal$.
Clearly $(Z,U)$ is point-transitive Hilbert-representable system
and, for convenience, we now rename it as $(X,U)$. 
Also, let us rename $\Hcal=Z_U(x_0)$, the cyclic space generated by 
$\{U^n x_0 : n \in \Z\}$, where $x_0$ is now some transitive point in $X$.

3.\ 
Let $X_{tr} \subset X$ be the dense $G_\del$ subset of the transitive
points in $X$. Similarly let $Y_{tr} \subset X$ be the dense $G_\del$ subset of 
the transitive points in $Y$. Clearly $\pi(X_{tr}) \subset Y_{tr}$ and we claim
that these sets are equal. 
Suppose $y_1 \in Y_{tr}$. Let $x_1 \in X$
be some point with $\pi(x_1) = y_1$. By assumption $y_1$
is a transitive point and there is a sequence $T^{n_i}y_1 \to y_0$.
We can assume that also $U^{n_i}x_1 \to x'_0$ for some point $x'_0 \in X$
with $\pi(x'_0) = y_0$. It follows that $\pi(\OC_U(x'_0)) = Y$
and therefore, by minimality, $\OC_U(x'_0) =X$. Thus $x'_0$ is a
transitive point and hence so is $x_1$ hence
$y_1 = \pi(x_1) \not\in X_{tr}$. We conclude that indeed $\pi(X_{tr}) = Y_{tr}$.

4.\ 
Next recall that with every recurrent-transitive metrizable WAP 
system $(W,R)$ there is an associated non-discrete Polish monothetic group 
$\La_W = {\text{unif-cls}}\{R^n: n \in \Z\} \subset \Homeo(W)$. 
Moreover, for every
transitive point $w_0 \in W$, the map $g \mapsto g w_0$ is a homeomorphism
from $\La_W$ onto the set $W_{tr}$ of transitive points of $W$,
see \cite{GW2}.
Now using this observations and with notations as in the previous step,
 we note that the surjection $\pi: X_{tr} \to Y_{tr}$ defines a surjective 
Polish group homomorphism $\rho: \La_X \to \La_Y$.
Moreover, we note that ${\rm ker}\,\rho = K$ is a compact subgroup of $\La_X$.

5.\ 
Set $\tilde{Y} = X/K$ and let $\sig : (X,U) \to (\tilde{Y}, \tilde{T})$ denote
the corresponding group homomorphism. 
We now obtained the diagram (\ref{AK}).
\end{proof}

\section{A Polish monothetic group $P$ which is reflexive but not
Hilbert-representable}

We will use Banaszczyk's construction of some families of Polish monothetic
strongly exotic groups \cite{Ba}. For our purposes it suffices to consider
one particular such example which we now proceed to describe.
Let $E=\ell_4(\N)$ be the Banach space comprising the real valued
sequences $u=\{x_j\}_{j=1}^\infty$ with 
$|u|:=(\sum_{j=1}^\infty x_j^4)^{1/4}  < \infty$.

Let $\{e_n\}_{n=1}^\infty$ be the standard basis of unit vectors and choose 
a dense sequence $\{a_n\}_{n=1}^\infty \subset E$ with the 
property $a_n \in \spann\{e_j\}_{j < n}$ for every $n \ge 1$.
Let $\Ga$ be the subgroup of $E$ generated by the sequence 
$\{e_n + a_n\}_{n=1}^\infty$.
It can be easily checked that  
$|\ga| \ge 1$ for every $0 \ne \ga \in \Ga$;
and that, on the other hand, $E = \Ga + 2 B_1(0)$,
where $B_1(0)$ denotes the unit ball in $E$. 
Set $P = E/\Ga$ and equip $P$ with its quotient topology.
It follows that $P$ is a Polish topological group. 
By Theorem 5.1 and Remark 5.2 in \cite{Ba} $P$ is strongly
exotic and monothetic. In Remark 5.2 Banaszczyk leaves the
proof of the fact that $P$ is monothetic as an exercise;
for completeness we provide a proof of this fact at the end of this section,
Proposition \ref{Mo}.

\begin{thm}\label{Ban}
The Polish monothetic group
$P = \ell_4(\N)/\Ga$ is reflexively-representable.
\end{thm}

\begin{proof}
We will follow the ideas of Megrelishvili's proof that the group $L_4([0,1])$
is reflexively-representable (\cite{Me-00}, Lemmas 3.3 and 3.4).
Note that Megrelishvili's proof works almost verbatim to 
show that $\ell_4(\N)$ is reflexively-representable.

We first define an invariant metric on $P$ by
$$
d(u+\Ga,0) = \|u+\Ga\| = \inf_{\ga \in \Ga} |u + \ga|,\qquad u \in E.
$$
Note that for $u \in E$ we have 
\begin{equation}\label{=}
|u| \le \frac12
\imp \|u + \Ga\| = |u|.
\end{equation}

Next define the following continuous function on $P$:
$$
F(u +\Ga) = \min\{\|u +\Ga\|, \frac{1}{10}\}.
$$
We will show that $F$ satisfies Grothendieck's {\em double limit condition},
i.e. we have to show that
given two sequences 
$\{u_m + \Ga\}_{m=1}^\infty, \{v_n + \Ga\}_{n=1}^\infty$ 
in $P$, if the limits
\begin{equation}\label{DL}
a=\lim_m\lim_n F(u_m + v_n + \Ga) \quad {\text{and}} \quad
b=\lim_n\lim_m F(u_m + v_n + \Ga),
\end{equation}
exist, then necessarily $a =b$.

Of course when $a=b = \frac{1}{10}$ we are done; so we now
assume that  $a < \frac{1}{10}$. Then for some $\del >0$ eventually 
$$
a_m := \lim_{n \to \infty} F(u_m +v_n +\Ga) =  
\lim_{n \to \infty} \|u_m + v_n + \Ga\| < \frac{1}{10}-2\del. 
$$
Thus for some $m_0$ we have $a_m  \le \frac{1}{10}-2\del$
for every $m \ge m_0$; and for a fixed $m \ge m_0$
there is an $n(m)$ such that for $n \ge n(m)$ we have
$\|u_m + v_n + \Ga\| \le \frac{1}{10}-\del$.

Now for $n \ge \max\{n(m), n(m_0)\}$ we have
$$
\|u_m + v_n + \Ga\| \le \frac{1}{10}-\del\quad {\text{and}} 
\quad  \|u_{m_0} + v_n + \Ga\| \le \frac{1}{10}-\del,
$$
hence $ \|u_m - u_{m_0} + \Ga\| \le \frac{2}{10}$.
Thus by an appropriate choice of representatives we can now assume that
$\|u_m - u_{m_0} + \Ga\| = |u_m - u_{m_0}| \le \frac{2}{10}$ for
all $m\ge m_0$. 
By the same token, for $n \ge n(m_0)$ we can assume that also
$\|u_{m_0} + v_n + \Ga\| = |u_{m_0} + v_n| \le \frac{2}{10} -\del$.
This means that for sufficiently large $m$ and $n$ all the vectors
$u_m$ and $-v_n$ lie in a ball of radius $\frac{2}{10}$ around $u_{m_0}$.
By the triangle inequality and (\ref{=})
\begin{equation*}
\|u_m + v_n +\Ga\| = |u_m + v_n| \le \frac{4}{10}-\del,
\end{equation*}
for all $m \ge m_0$ and $n \ge n(m_0)$.

Now in $E$ we have:
\begin{gather*}
|u_m + v_n|^4 
= \sum_{i=1}^\infty u_m(i)^4 + 4u_m(i)^3v_n(i) + 6u_m(i)^2v^2_n(i)
+ 4u_m(i)v_n(i)^3 + v_n(i)^4.
\end{gather*}
As all these vectors lie in a ball we can assume, by passing to 
subsequences, that 
\begin{align*}
& \lim_{n \to \infty} v_n = v \quad {\text{weakly in}}  \ \ell_4\\ 
& \lim_{n \to \infty}v^2_n = v^2 \quad {\text{weakly in}} \ \ell_2\\ 
& \lim_{n \to \infty}v^3_n = v^3 \quad {\text{weakly in}} \  \ell_{\frac{4}{3}},
\end{align*}
whence
$$
a_m = |u_m|^4 + 4\langle u_m^3,v \rangle + 6 \langle u_m^2,v^2 \rangle
+4 \langle u_m,v^3 \rangle + V^4,
$$
where 
$V = \lim_{n \to \infty}|v_n| 
$.

Again we can assume that also
\begin{align*}
& \lim_{m \to \infty} u_m = u \quad {\text{weakly in}}  \ \ell_4\\ 
& \lim_{m \to \infty}u^2_n = u^2 \quad {\text{weakly in}} \ \ell_2\\ 
& \lim_{m \to \infty}u^3_n = u^3 \quad {\text{weakly in}} \  \ell_{\frac{4}{3}}.
\end{align*}
We now get
$$
a = \lim_{m\to\infty}a_m = 
U^4 + 4\langle u^3,v \rangle + 6 \langle u^2,v^2 \rangle
+4 \langle u,v^3 \rangle + V^4,
$$
where 
$U = \lim_{n \to \infty}|u_n|$.

Computing the double limit the other way we have similarly
$$
b =
U^4 + 4\langle u^3,v \rangle + 6 \langle u^2,v^2 \rangle
+4 \langle u,v^3 \rangle + V^4,
$$
whence $a =b$, as required.

We have thus shown that the function $F$ satisfies the double limit condition.
Now observe that, by the same computations, 
for any fixed $v \in E$ the function $F_v(u) = F(u - v)$
also satisfies the double limit property. 
By the double limit criterion \ref{DLC} we conclude that
the family $F_v, v \in E$, is contained in $WAP(P)$.
As clearly the collection $\{F_v\}_{v \in E}$ generates
the topology on $P$ it follows that
this family separates points and closed sets on $P$ 
and we conclude from Theorem \ref{S-M}.(1), as in \cite{Me-00}, 
that the group $P$ is indeed reflexively-representable. 
\end{proof}

\begin{thm}\label{WAPH-1}
There is a Polish monothetic group $P$ which admits
a faithful representation on a reflexive Banach space
but is not Hilbert-representable (in fact it is strongly exotic).
\end{thm}

\begin{proof}
Let $P= \ell_4(\N)/\Ga$ be the strongly exotic Polish monothetic group 
described by Banaszczyk \cite{Ba}. By Theorem \ref{Ban}
$P$ is reflexively-representable and, being strongly exotic, it is not 
Hilbert-representable.
\end{proof}

\begin{prop}\label{Mo}
The group $P$ is monothetic.
\end{prop}

\begin{proof}
Let $\pi_k: \ell_4(\N) \to \ell^k_4:=E_k$ be the projection map
onto the first $k$ coordinates. Let $\Ga_k = \pi_k(\Ga)$ and $P_k = E_k/\Ga_k$.
Note that for each $k \ge 1$ the group $P_k$ is a $k$-dimensional torus
and by Kronecker's theorem the vectors $z=(z_1,\dots,z_k) \in E_k$ such that
$\{nz +\Ga_k: n \in \Z\}$ is dense in $P_k$, form a dense subset of $E_k$.
We will construct, by induction, a Cauchy sequence $\{z^{(k)}: k \in \N\}$
in $\ell_4(\N)$ whose limit $x =(x_1,x_2,\dots) \in \ell_4(\N)$ will have
the property that $\{nx +\Ga: n \in \Z\}$ is a dense subgroup of $P$.

Suppose we have already chosen $w^{(k)} =(z^{(k)}_1,z^{(k)}_2,\dots,z^{(k)}_k) 
\in E_k$  with the property that for some positive integer $N_k$
\begin{equation}\label{Nk}
\{nw^{(k)} +\Ga_k: |n| \le N_k\},
\end{equation} 
is $\frac{1}{k}$-dense in $P_k$. 
We set 
$$
z^{(k)}=(z^{(k)}_1,z^{(k)}_2,\dots,z^{(k)}_k,0,0,\dots)\in \ell_4(\N).
$$
There exists an $\ep_k >0$ such that the condition (\ref{Nk})
(with $\pi_k(z)$ instead of $w^{(k)}$)
is satisfied for every $z$ in an $\ep_k$-ball around $z^{(k)}$. 
As we have seen there is then a vector 
$w^{(k+1)} =(z^{(k+1)}_1,z^{(k+1)}_2,\dots,z^{(k+1)}_{k+1}) 
\in E_{k+1}$ and a positive integer $N_{k+1}$ such that the set
$\{nw^{(k+1)} +\Ga_{k+1}: |n| \le N_{k+1}\}$ 
is $\frac{1}{k+1}$-dense in $P_{k+1}$,
and such that for 
$z^{(k+1)}=(z^{(k+1)}_1,z^{(k+1)}_2,\dots,z^{(k+1)}_k,0,0,\dots)\in \ell_4(\N)$ 
we have,
$|z^{(k)} - z^{(k+1)}| < \min\{2^{-k},\ep_k\}$. 
This completes the inductive construction and clearly the unique limit point
$x = \lim_{k \to \infty} z^{(k)}$ is the required one.
\end{proof}

\section{A WAP recurrent-transitive system which is not Hilbert}

\begin{thm}\label{WAPH-2}
The Banaszczyk group $P=\ell_4/\Ga$ admits no nontrivial Hilbert system. 
Thus any nontrivial WAP action of $P$ provides an example of 
a recurrent-transitive WAP system which is not Hilbert.
\end{thm}

\begin{proof}
1.\  
Let $T\in P$ be a generator of a dense cyclic subgroup in $P$.
Suppose that $P$ admits a nontrivial point-transitive Hilbert 
system $(Y,y_0,P)$ and consider the $\Z$-system $(Y,y_0,T)$.
Since $P$ is exotic it is in particular MAP 
(= minimally almost periodic, i.e. it admits no nontrivial
unitary finite dimensional representations) and it follows 
that $y_0$ is a non-periodic, recurrent point of $(Y,T)$.

Let $(X,x_0,U)$ with $x_0$ a unit vector in a Hilbert space $\Hcal$,
$U \in \Ucal(\Hcal)$ a unitary operator on $\Hcal$,
and $X = w{\text{-}}\cls \{U^nx_0 : n \in \Z\}$, be the extension provided by 
Step 5 of Theorem \ref{AK-ext} (diagram  (\ref{AK}) above).

Let $G_X = \cls \{U^n : n \in \Z\}$, where the closure is taken in the strong
operator topology on $\Ucal(\Hcal)$. It is not hard to see that
there is a canonical topological isomorphism between the groups $\La_X$
and  $G_X$.
There are also canonical
topological isomorphisms $\La_X/K \cong \La_Y \cong \La_{\tilde{Y}}$.
Note that our assumption that $(Y,T)$ extends to a $P$-system means that
the map $n \mapsto T^n$ extends to a continuous homomorphism
$P \to \La_Y$ with a dense image (see Step 4 of the proof of Theorem 
\ref{AK-ext}).

By Proposition \ref{sh=pd} we have an isomorphism
$(X,U)\cong (X_f,S)$ where $f$ is the positive definite function
$f(n) = \langle U^n x_0,x_0 \rangle$. Now let 
$(\Om, \Fcal, \mu, R)$ be the Gauss dynamical system which 
corresponds to $f$; thus there is a $\Z$-sequence $\{\xi_n\}_{n \in \Z}$ 
of real valued random variables defined on $\Om$ with joint
Gauss distribution such that $\xi_n = \xi_0 \circ R^n$
with $\E(\xi_n\xi_m) = f(n -m)$, and such that
the $\sig$-algebra generated by the sequence $\{\xi_n\}_{n \in \Z}$
is $\Fcal$. (For more details on Gauss dynamical systems see
e.g. \cite{CFS} and \cite{LPT}.)

Let $\Hcal_0 \subset L_2(\Om,\mu)$ be the closed subspace
spanned by $\{\xi_n\}_{n \in \Z}$ (the first Wiener chaos). Then
$\Hcal_0$ is isomorphic, as a Hilbert space, to $\Hcal$
(which is the $U$-cyclic space spanned by $x_0$).  In particular
we can think of $\La_X$ as a subgroup of $\Ucal_0(\Hcal_0)$ where it 
commutes with $U_R$, the Koopman operator on $L_2(\Om,\mu)$
which corresponds to $R$.
It follows that $\La_X$ can be realized a a group of measure
preserving transformations (i.e. a subgroup of $\Aut(\Om,\mu)$;
see e.g. \cite{LPT}).
In particular we can now view $K$ as a compact subgroup of 
$\Aut(\Om,\mu)$ and then define the compact group-factor map
$\pi_K : (\Om, \Fcal, \mu, R) \to (\Om/K, \Fcal/K, \mu_K, R_K)$.

Finally considering $L_2(\Om/K,\mu_K)$ we see that the Polish group
$\La_X/K$ is faithfully represented on this Hilbert space.
Since we have a continuous homomorphism $P \to 
\La_Y \cong \La_X/K$ with a dense image, this contradicts
the fact 
that $P$ is exotic and our proof is complete.
\end{proof}

\begin{cor}\label{rf}
There exist a recurrent-transitive WAP function 
$f \in WAP(\Z) \setminus H(\Z)$.
\end{cor}

\begin{proof}
Let $P$ be the Polish monothetic exotic Banaszczyk group.
Let $T\in P$ be a generator of a dense cyclic subgroup in $P$.
As we have seen in Theorem \ref{Ban}, the functions $\{F_v\}_{ v \in E}$
are in $WAP(P)$ and they separate points and closed sets on $P$.
Therefore their restrictions $f_v= F_v \rest \Z$, to the dense subgroup  
$\{T^n : n \in \Z\} \subset P$, are in $WAP(\Z)$. As elements of $WAP(P)$
the functions $F_v$ are recurrent and hence so are the functions $f_v$,\
$v \in E$.
Finally by Theorem \ref{WAPH-2} each $f_v$ is an element of 
$WAP(\Z) \setminus H(\Z)$.
\end{proof}

\begin{cor}\label{AK-ext-cor}
If a Polish monothetic group $P$ is Hilbert-representable
and if $K \subset P$ is a compact subgroup, then
the quotient group $P/K$ is Hilbert-representable.
\end{cor}

\begin{proof}
Follow the last steps of the proof of Theorem \ref{WAPH-2}
with $P$ taking the place of $\La_X$.
\end{proof}

\section{Recurrent functions in $H(\Z) \setminus B(\Z)$}\label{R-H}

\begin{lem}\label{Rec-WAP}
For every recurrent $f \in WAP(\Z)$ the system $(X_f,S)$
is uniformly rigid and the $\Z$-algebra $\Acal_f$
consists entirely of recurrent functions.
\end{lem}

\begin{proof}
The dynamical system $X_f=
\cls \{S^n f: n \in \Z\}$ (in the pointwise convergence topology) 
is WAP and therefore uniformly rigid.
It then follows that the non-discrete Polish monothetic group
$\La_{X_f}= \cls \{S^n : n \in \Z\}$ (in the uniform convergence topology
on $\Homeo(X_f)$) acts on $X_f$ as a group of automorphisms.
In turn this implies that every function $g \in \Acal(X_f,f) = \Acal_f$
is a recurrent point: $g = \lim_{k \to \infty} S^{n_k}g$, for every sequence
$n_k \nearrow\infty$ with $ S^{n_k}$ tending to the identity in $\La_{X_f}$. 
\end{proof}

\begin{lem}\label{dense}
If $\Acal$ is an infinite-dimensional $\Z$-subalgebra of $H(\Z) = \ol{B(\Z)}$
then $\Acal\setminus B(\Z)$ is a (norm) dense Borel subset of $\Acal$.
\end{lem}

\begin{proof}
The Fourier transform $\Fcal : M(\T) \onto B(\Z)$ is a 1-1 onto
bounded linear map ($\|\Fcal(\mu)\|_\infty \le \|\mu\|$).
If $\Acal \subset \Fcal(M(\T)) = B(\Z)$ then $\Fcal^{-1}\Acal \subset M(\T)$
is an infinite-dimensional $C^*$-algebra, which is impossible
(see e.g. Theorem 2.1 in \cite{DR}).
Thus $\Acal \subsetneq B(\Z)$.
Moreover, if the algebra $\Acal \cap B(\Z)$
contains a nonempty $\Acal$-norm-open subset then it must coincide
with $\Acal$ which again is impossible. Thus $\Acal \setminus B(\Z)$
is norm-dense in $\Acal$.

Finally the algebra $B(\Z)$, being a 1-1 continuous image of $M(\T)$
(under $\Fcal$), is a Borel subset of $H(\Z) = \ol{B(\Z)}$, and hence so are
the sets $\Acal \cap B(\Z)$ and $\Acal \setminus B(\Z)$.
\end{proof}

\begin{cor}
In every infinite dimensional $\Z$-subalgebra $\Acal \subset H(\Z)$
the set $\Acal_{rec} \setminus B(\Z)$, where $\Acal_{rec}$ is the collection
of recurrent functions in $\Acal$, is a (norm) dense subset of
$\Acal_{rec}$. 
In particular there are recurrent functions in $H(\Z)$ which are not
Fourier-Stieltjes transforms (i.e. are not in $B(\Z)$). 
\end{cor}

\begin{proof}
Combine Lemmas \ref{Rec-WAP} and \ref{dense}.
\end{proof}

\section{A structure theorem for weakly almost periodic systems}

\begin{thm}\label{AK-ext-WAP}
Let $(X,T)$ be a metrizable recurrent-transitive WAP dynamical system
and $\pi: (X,T) \to (Y,S)$ a factor.
Then there is an almost 1-1 extension which is a compact group-factor of a 
recurrent-transitive subsystem $Z \subset X$. 
More explicitly there is a commutative diagram
\begin{equation*}\label{di}
\xymatrix
{
(Z,T)  \ar[dd]_{\pi}\ar[dr]^{\sig}  & \\
& (\tilde{Y},\tilde{T})\ar[dl]^{\rho}\\
(Y,S) &
}
\end{equation*}
with $Z \subset X$ a subsystem, $\sig$ a group-extension and 
$\rho$ an almost 1-1 extension.
\end{thm}

\begin{proof}
Just repeat the steps 2. to 5. of the proof of Theorem \ref{AK-ext} above.
\end{proof}

\section{Some open problems}

\begin{prob}\label{abc}
In Theorem \ref{AK-ext} we have shown that every metrizable 
recurrent-transitive Hilbert system admits an almost 
1-1 extension which is a group-factor of a Hilbert-representable system. 
Can one get rid in this structure
theorem of either one of these extensions or maybe of both?
Thus, our question is whether a metrizable recurrent-transitive 
Hilbert system is always: (a) Hilbert-representable,
or (b) a group-factor of a Hilbert-representable system, or (c)
an almost 1-1 factor of a Hilbert-representable system.
Question (a) can be reformulated as follows: Is a factor of a 
Hilbert-representable system also Hilbert-representable? 
This latter question is stated as Problem 996 in \cite{M-O}.
\end{prob}

\begin{prob}
Is there a metrizable recurrent-transitive Hilbert system $(X,T)$ with no 
nontrivial Hilbert-representable factors? In other words, is there a
nontrivial $\Z$-algebra $\Acal \subset H(\Z)$ with 
$\Acal \cap B(\Z) = \C$? Such a system will provide a counter-example to 
option (a) in Problem \ref{abc}. 
\end{prob}

\end{document}